\documentclass[11pt,reqno]{amsart}

\usepackage[utf8]{inputenc}
\usepackage[T1]{fontenc}
\usepackage{amsmath,amssymb,amsthm}
\usepackage{mathrsfs}
\usepackage[english]{babel}
\usepackage[a4paper,margin=3.2cm]{geometry}
\usepackage{xcolor}
\usepackage[colorlinks=true,linkcolor=blue!60!black,citecolor=blue!60!black,urlcolor=blue!60!black]{hyperref}

\theoremstyle{plain}
\newtheorem{theorem}{Theorem}[section]
\newtheorem{lemma}[theorem]{Lemma}
\newtheorem{proposition}[theorem]{Proposition}
\newtheorem{corollary}[theorem]{Corollary}
\newtheorem{conjecture}[theorem]{Conjecture}
\newtheorem*{theoremA}{Theorem A}
\newtheorem*{theoremK}{Theorem (Kocsard \cite{Koc13})}

\theoremstyle{definition}
\newtheorem{definition}[theorem]{Definition}
\newtheorem{question}[theorem]{Question}
\newtheorem{construction}[theorem]{Construction}

\theoremstyle{remark}
\newtheorem{remark}[theorem]{Remark}

\newcommand{\R}{\mathbb{R}}
\newcommand{\Z}{\mathbb{Z}}
\newcommand{\T}{\mathbb{T}}
\newcommand{\N}{\mathbb{N}}
\newcommand{\Q}{\mathbb{Q}}
\newcommand{\Homeo}{\operatorname{Homeo}}
\newcommand{\Diff}{\operatorname{Diff}}
\newcommand{\Stab}{\operatorname{Stab}}
\newcommand{\Inv}{\operatorname{Inv}}

\newcommand{\CzM}{C^{0}(M)}
\newcommand{\coc}{Z^{1}(\alpha)}
\newcommand{\cob}{B(\alpha)}
\newcommand{\norm}[1]{\left\|#1\right\|}
\newcommand{\normz}[1]{\left\|#1\right\|_{C^{0}}}
\newcommand{\e}{\mathbf{e}}
\DeclareMathOperator{\diam}{diam}
\DeclareMathOperator{\dist}{dist}

\begin{document}

\title[Classification of cohomologically stable actions]{Classification of some cohomologically $C^0$ - stable continuous group actions on metric spaces}
\author{Boris Petkovi\'c}
\address{University of Banja Luka \\ Mladena Stojanovi\'ca 2
78 000 Banja Luka
Bosnia and Herzegovina}
\email{boris.petkovic@pmf.unibl.org}


\begin{abstract}
After Katok \cite{Kat01}, a homeomorphism $f\colon M\to M$ of a compact
metric space is said to be cohomologically $C^0$-stable if its
space of real $C^0$-coboundaries is closed in $C^0(M)$. Kocsard
\cite{Koc13} proved that this is the case if and only if $f$ is periodic. We extend the classification to actions of arbitrary finitely generated groups: an action $\alpha\colon G\to\Homeo(M)$ is
cohomologically $C^0$ - stable if and only if the image $\alpha(G)$ is a finite group. In particular this settles the case of $\Z^k$-actions
generated by finitely many commuting homeomorphisms. Notably, no
amenability assumption is needed: we explain why spectral-gap phenomena for non-amenable actions, which do produce cohomological stability in H\"older, Sobolev and $L^2$ categories, are invisible to the uniform norm. We also discuss the genuinely different smooth category and state a conjecture regarding cohomological $C^\infty$-stability of $\Z^k$-action by smooth circle diffeomorphisms without
periodic orbits, connecting the problem with works of Moser, Fayad-Khanin, Avila-Kocsard and
Petkovi\'c.
\end{abstract}

\maketitle

\section{Introduction}\label{sec:intro}

Cocycles and cohomological equations are ubiquitous in dynamics and
ergodic theory. We refer interested readers to Katok's survey \cite{Kat01} for a panorama. Throughout this note $(M,d)$ denotes a compact metric space and $\CzM$ the Banach space of real continuous functions on $M$ with the uniform norm $\normz{\phi}:=\sup_{x\in M}|\phi(x)|$. \\

For a single homeomorphism $f\colon M\to M$, a (real, continuous)
cocycle is simply a function $\phi\in\CzM$, and $\phi$ is a
coboundary if the cohomological equation
\[
\phi = u\circ f - u
\]
admits a solution $u\in\CzM$. Writing
$B(f,\CzM):=\{v\circ f-v : v\in\CzM\}$, the homeomorphism $f$ is called
cohomologically $C^0$-stable when $B(f,\CzM)$ is closed in
$\CzM$. Stability is precisely the property that makes the classical
obstructions (in this case $f$-invariant probability measures) a complete criterion for solvability. Kocsard gave
the following complete classification.

\begin{theoremK}
A homeomorphism $f\colon M\to M$ is cohomologically $C^0$-stable if and
only if $f$ is periodic, i.e.\ it has finite order in $\Homeo(M)$.
\end{theoremK}

The purpose of this note is to observe that the classification extends to actions of arbitrary finitely generated groups,
with a proof in the spirit of \cite{Koc13} but with one new geometric
ingredient. Given a group $G$ and an action $\alpha\colon G\to\Homeo(M)$ we call $\alpha$ periodic when its
image $\alpha(G)\leq\Homeo(M)$ is a finite group. Cocycles,
coboundaries and cohomological $C^0$-stability for $\alpha$ are defined in Section 2. For $G=\Z$ all notions reduce to the classical
ones.

Our main theorem is the following one.

\begin{theoremA}
Let $G$ be a finitely generated group and let
$\alpha\colon G\to\Homeo(M)$ be an action on a compact metric space
$M$. Then $\alpha$ is cohomologically $C^0$-stable if and only if
$\alpha$ is periodic, i.e.\ $\alpha(G)$ is finite.
\end{theoremA}

As a direct consequence of Theorem A we get the following corollary. 

\begin{corollary}\label{cor:Zk}
An action $\alpha\colon\Z^k\to\Homeo(M)$ generated by commuting
homeomorphisms $f_1,\dots,f_k$ is cohomologically $C^0$-stable if and
only if some finite-index sublattice $\Lambda\leq\Z^k$ acts trivially,
i.e.\ the group $\langle f_1,\dots,f_k\rangle\leq\Homeo(M)$ is finite.
\end{corollary}

For $k=1$, in Theorem A, one recovers precisely Kocsard's theorem. Note that for $k\geq 2$ the correct dichotomy is finiteness of the image since the individual
generators may well be periodic while the action is unstable
(e.g. two rotations of finite but coprime orders generate a finite
group - stable, while a rational and an irrational rotation
generate an infinite one - unstable). \\

Two features of Theorem A deserve emphasis, therefore for each of them we rdevote a
section. \\

First, no amenability hypothesis appears. This may be
surprising since the instability half of the Theorem A in the proof requires functions that are simultaneously almost invariant under all generators, and the existence of almost invariant objects is usually an amenability phenomenon. The resolution is that almost invariance here is measured in the uniform norm, where the so called tent functions of the word metric
on large orbits are slowly varying on every finitely generated
group. Expansion and spectral gaps are $\ell^2$-obstructions and are
invisible to $\ell^\infty$. However, amenability does play a role in this theorem, but in the description of the
closure of the coboundary space (see Section \ref{sec:amenable}), where
invariant measures (guaranteed to exist exactly in the amenable
case) provide the natural obstructions, following \cite{MOP77}.
In Section \ref{sec:gap} we discuss in detail why natural non-amenable candidates for stable infinite actions (such as free subgroups of
$SO(3)$ acting on $S^2$ with spectral gap \cite{Dri84,BG08}) fail to be counterexamples, and in which function spaces the gap does
produce closed ranges. \\

Second, Theorem A confirms that the $C^0$-category is arithmetically soft. Stability is decided by a crude algebraic invariant (finiteness of the image) and no Diophantine-type condition can be detected. This is in sharp contrast with the smooth category, where,
for a single minimal circle diffeomorphism, Avila and Kocsard
\cite{AK11} proved that cohomological $C^\infty$-stability holds if and only if the rotation number is Diophantine. In Section \ref{sec:smooth}, in the form of a conjecture, we formulate the natural $\Z^k$-analogue: cohomological $C^\infty$-stability of a $\Z^k$-action by circle diffeomorphisms without periodic orbits should be equivalent to the simultaneous
Diophantine condition (SDC) on the rotation vector, the condition
introduced by Moser \cite{Mos90} and shown by Fayad and Khanin
\cite{FK09} to imply smooth simultaneous linearization. We explain why one implication of the conjecture follows from \cite{FK09} together with a Fourier computation for the linear model, why the converse is open precisely for non-linearizable actions with non-SDC rotation vectors, and how the question connects with  perturbative classification of Diophantine $\Z^m$-actions on higher-dimensional tori in \cite{P21}.

\subsection*{Organization of the paper.} Section \ref{sec:prelim} sets notation.
Section \ref{sec:stable} proves the easier half of Theorem A (periodic $\Rightarrow$ stable) by finite-group averaging. Section \ref{sec:unstable} proves the main half of Theorem A by constructing word-ball tent functions. Section \ref{sec:amenable} discusses the closure of the coboundary space and
the true role of amenability. Section \ref{sec:gap} discusses spectral gaps and  why natural non-amenable candidates for stable infinite actions fail to be counterexamples. Section \ref{sec:smooth} treats the $C^0$-versus-$C^\infty$ dichotomy and states the  cohomological $C^\infty$ - stability conjecture. Last Section \ref{sec:questions} collects open questions.

\section{Notation and preliminaries}\label{sec:prelim}

Let $G$ be a group generated by a finite symmetric set
$S=S^{-1}$ with $e\notin S$, let $|g|=|g|_S$ denote the word length metric, and let $\alpha\colon G\to\Homeo(M)$ be an action on the compact metric space $(M,d)$. For $x \in M$ and $g \in G$ we write $g\cdot x$ instead of $\alpha(g)(x)$.

\begin{definition}\label{def:cocycle}
A (real, continuous) cocycle over $\alpha$ is a map
$c\colon G\times M\to\R$ such that $c(g,\cdot)\in\CzM$ for every
$g\in G$ and
\begin{equation}\label{eq:cocycle}
c(gh,x) = c(g,h\cdot x) + c(h,x),
\qquad \forall\, g,h\in G,\ \forall\, x\in M.
\end{equation}
The linear space of cocycles is denoted by $\coc$. Every $u\in\CzM$
determines the coboundary $\delta u\in\coc$ defined as
\[
(\delta u)(g,x) := u(g\cdot x)-u(x),
\]
and we set $\cob:=\delta\bigl(\CzM\bigr)\subset\coc$.
\end{definition}

We note that a cocycle is determined by its values on the generating set $S$ since \eqref{eq:cocycle} gives $c(e,\cdot)\equiv 0$,
$c(g^{-1},x)=-c(g,g^{-1}\cdot x)$, and the value on any word in $S$ is
a finite sum of translates of the generator values. In particular
\begin{equation}\label{eq:growth}
\normz{c(g,\cdot)} \le |g|\,\max_{s\in S}\normz{c(s,\cdot)},
\qquad \forall g\in G .
\end{equation}
We therefore topologize $\coc$ by the norm
\[
\norm{c} := \max_{s\in S}\ \normz{c(s,\cdot)} .
\]
Different finite generating sets give equivalent norms, by
\eqref{eq:growth}.

\begin{remark}\label{rem:closed-in-product}
The evaluation map $c\mapsto (c(s,\cdot))_{s\in S}$ embeds $\coc$
isometrically as a closed linear subspace of $\CzM^{S}$ (with
the max norm) since the constraints imposed on a tuple
$(\phi_s)_{s\in S}$ by the relations of $G$ through \eqref{eq:cocycle}
are finite linear combinations of compositions with fixed
homeomorphisms, hence closed conditions. Consequently
$(\coc,\norm{\cdot})$ is a Banach space. For $G=\Z^k$ with standard
generator set $S=\{\e_1, \ldots, \e_k$\}, $\coc$ is exactly the space of tuples
$(\phi_1,\dots,\phi_k)\in\CzM^k$ satisfying the compatibility
relations
\[
\phi_i\circ f_j - \phi_i \;=\; \phi_j\circ f_i-\phi_j,
\qquad 1\le i<j\le k,
\]
where $f_i:=\alpha(\e_i)$, and $\delta u = (u\circ f_1-u,\dots,
u\circ f_k-u)$.
\end{remark}

The kernel of $\delta$ is the space
$\Inv(\alpha):=\{v\in\CzM : v\circ\alpha(g)=v,\ \forall g\in G\}$ of
continuous $\alpha$-invariant functions. Every $v\in\Inv(\alpha)$ is
constant on each orbit and hence on each orbit closure. Note that
$\norm{\delta u}\le 2\normz{u}$, so
$\delta\colon\CzM\to\coc$ is a bounded operator.

\begin{definition}
The action $\alpha$ is cohomologically $C^0$-stable when $\cob$
is closed in $\coc$.
\end{definition}

Previous definition is equivalently to $\cob$ beeing closed in $\CzM^S$) by Remark~\ref{rem:closed-in-product}. \\

In the special case $G=\Z$ this coincides with the definition in
\cite{Kat01} and \cite{Koc13}. \\

Given $x\in M$, the Schreier graph of the orbit
$G\cdot x$ has vertex set $G\cdot x$ and an edge between $y$ and
$s\cdot y$ for each $y\in G\cdot x$ and $s\in S$. We write
$d_{\mathrm{orb}}$ for the associated graph metric. The graph is connected and of degree at most $|S|$, and
\begin{equation}\label{eq:orb-vs-word}
d_{\mathrm{orb}}(x, g\cdot x)\ \le\ |g|,\end{equation}
with equality realized: every point at graph
distance $D$ from $x$ is of the form $g\cdot x$ with $|g|=D$.

\section{Periodic actions are stable}\label{sec:stable}

The following lemma extends Lemma 3.1 from \cite{Koc13} and identifies the coboundary space of a periodic action explicitly. The averaging trick used in the proof is the group-action analogue of a formula from \cite{MOP77}.

\begin{lemma}\label{lem:finite-stable}
Assume $\Gamma:=\alpha(G)$ is finite and let $K:=\ker\alpha
=\{g\in G:\alpha(g)=\mathrm{id}_M\}$. Then
\begin{equation}\label{eq:char-finite}
\cob \;=\; \bigl\{\, c\in\coc \;:\; c(n,\cdot)\equiv 0
\ \ \forall\, n\in K \,\bigr\}.
\end{equation}
In particular $\cob$ is closed, i.e.\ $\alpha$ is cohomologically
$C^0$-stable.
\end{lemma}

\begin{proof}
If $c \in \cob$ then $c=\delta u$ for some contiouns function $u$ and if $n\in K$ then
$c(n,x)=u(n\cdot x)-u(x)=u(x)-u(x)=0$ for all $x \in M$. \\

Conversely, let $c\in\coc$ vanish on $K$. We first claim that
$c(g,x)$ depends only on the coset $gK$, i.e. only on
$\gamma=\alpha(g)\in\Gamma$. For $n\in K$,
\[
c(gn,x)=c(g,n\cdot x)+c(n,x)=c(g,x),
\qquad
c(ng,x)=c(n,g\cdot x)+c(g,x)=c(g,x),
\]
using $n\cdot y=y$ and $c(n,\cdot)\equiv 0$. Hence $c$ descends to a
cocycle $\bar c\colon\Gamma\times M\to\R$ over the (tautological)
action of the finite group $\Gamma$. Define
\[
u(x) := -\frac{1}{|\Gamma|}\sum_{\gamma\in\Gamma}\bar c(\gamma,x),
\qquad x\in M,
\]
which is continuous. For every $h\in\Gamma$, using
\eqref{eq:cocycle} in the form
$\bar c(\gamma,h\cdot x)=\bar c(\gamma h,x)-\bar c(h,x)$ and
reindexing $\gamma h\mapsto\gamma$:
\[
u(h\cdot x)-u(x)
= -\frac{1}{|\Gamma|}\sum_{\gamma\in\Gamma}
\bigl[\bar c(\gamma h,x)-\bar c(h,x)-\bar c(\gamma,x)\bigr]
= \bar c(h,x).
\]
Thus $c=\delta u$. \\

Let us now prove the closedness part of the lemma. Since $[G:K]=|\Gamma|<\infty$ and $G$ is finitely
generated, $K$ is finitely generated. 
Let $n_1,\dots,n_m$ be generators of $K$. If $c$ vanishes
on the $n_i$ then it vanishes on all of $K$, by \eqref{eq:cocycle}
and induction on word length in the $n_i^{\pm1}$. By
\eqref{eq:growth}, each linear map $c\mapsto c(n_i,\cdot)\in\CzM$ is
bounded on $\coc$. Therefore the right-hand side of
\eqref{eq:char-finite} is an intersection of kernels of finitely many
bounded operators, so it is closed.
\end{proof}

\begin{remark}
For $G=\Z^k$, \eqref{eq:char-finite} reads: $(\phi_1,\dots,\phi_k)$ is
a coboundary if and only if the Birkhoff sums associated with any basis $n_1,\dots,n_k$ of the finite-index sublattice $K=\ker\alpha$ vanish
identically. For $k=1$, $K=q\Z$ and one recovers exactly the
characterization $\sum_{j=0}^{q-1}\phi(f^j(x))=0$ in Lemma 3.1 of
\cite{Koc13}.
\end{remark}

\section{Non-periodic actions are unstable}\label{sec:unstable}

We now prove the main implication of Theorem A. The functional-analytic part is identical to that one in Lemma 3.2 of \cite{Koc13}.

\begin{lemma}\label{lem:open-mapping}
If $\cob$ is closed in $\coc$, then there exists $C=C(\alpha)>0$ such
that
\begin{equation}\label{eq:gap-ineq}
\dist_{C^0}\bigl(u,\Inv(\alpha)\bigr)
\;\le\; C\,\norm{\delta u}
\;=\; C\,\max_{s\in S}\ \normz{\,u\circ\alpha(s)-u\,},
\qquad \forall\, u\in\CzM .
\end{equation}
\end{lemma}

\begin{proof}
The quotient $C^0_\alpha(M):=\CzM/\Inv(\alpha)$ is a Banach space with
the quotient norm, and the factor operator
$\bar\delta\colon C^0_\alpha(M)\to\cob$, $\bar\delta(u+\Inv(\alpha))
:=\delta u$, is well defined, bounded (with norm $\le 2$, as in
\cite[Lemma 3.2]{Koc13}) and bijective onto $\cob$. If $\cob$ is
closed in the Banach space $\coc$, then $\cob$ is itself Banach and,
by the open mapping theorem, $\bar\delta$ is an isomorphism of Banach
spaces. Now \eqref{eq:gap-ineq} is the boundedness of
$\bar\delta^{\,-1}$.
\end{proof}

For every non-periodic action cohomological instability will follow precisely from the construction of functions violating condition \eqref{eq:gap-ineq}. We first prove two lemmas concerning unbounded orbits of group actions with infinite image, and their diameters.

\begin{lemma}\label{lem:orbits}
If $\alpha(G)$ is infinite, then $\sup_{x\in M}\,\#(G\cdot x)=\infty$.
\end{lemma}

\begin{proof}
Suppose $\#(G\cdot x)\le N$ for all $x\in M$. Then every stabilizer
$\Stab(x)$ has index at most $N$ in $G$. A subgroup of index
$n\le N$ arises as a point stabilizer of a transitive $G$-action on
$\{1,\dots,n\}$, i.e.\ from a homomorphism $G\to\mathrm{Sym}(n)$.
Since $G$ is finitely generated there are only finitely many such
homomorphisms, hence only finitely many subgroups of index $\le N$.
Their intersection $\Lambda\le G$ has finite index and is contained in
$\Stab(x)$ for every $x$, so $\alpha(\Lambda)=\{\mathrm{id}_M\}$ and
$\alpha(G)$ is a quotient of the finite group $G/\Lambda$ which is a contradiction.
\end{proof}

\begin{lemma}\label{lem:diam}
If $\alpha(G)$ is infinite, then for every $R\in\N$ there exists
$x_0\in M$ whose orbit Schreier graph has diameter at least $R+1$. Moreover, there exists $z^*\in G\cdot x_0$ with
$d_{\mathrm{orb}}(x_0,z^*)= R+1$.
\end{lemma}

\begin{proof}
Balls in a connected graph of degree $\le|S|$ satisfy
$\#B_{\mathrm{orb}}(x,D)\le (|S|+1)^{D}$. If some orbit is infinite, its
Schreier graph is infinite, connected and locally finite, hence has
points at every finite distance from $x_0$. If all orbits are finite,
Lemma~\ref{lem:orbits} provides orbits of cardinality
$>(|S|+1)^{R+1}$, forcing diameter $\ge R+1$ around any base point.
In either case, connectedness yields a point $z^*$ at graph distance
exactly $R+1$ from $x_0$ (truncate a geodesic to a farther point).
\end{proof}

The new ingredient in the proof is the following word - ball "tent" construction. It replaces
the sum of bumps spread along a $\Z$-orbit segment used in
\cite[eq.~(2)]{Koc13} by a maximum of bumps transported along a
word-metric ball. Using $\max$ instead of $\sum$ makes all
disjointness issues disappear (in particular those caused by point stabilizers,
which are absent when $k=1$ but unavoidable for general $G$).

\begin{construction}\label{con:tent}
Fix $R\in\N$, $R\ge 2$. Let $x_0$ and $z^*$ be as in
Lemma~\ref{lem:diam}. All the orbit points relevant below lie in the
finite set
\[
E_R := \{\, h\cdot x_0 \;:\; h\in G,\ |h|\le 2R \,\} \subset M ,
\]
and we let
\[
\varepsilon_R := \min\{\, d(p,q) \;:\; p\ne q,\ p,q\in E_R \,\} > 0,
\qquad
0 < r < \varepsilon_R .
\]
Define $\tau\colon G\to[0,1]$ by $\tau(g):=\max\{0,\,1-|g|/R\}$, so
$\tau(e)=1$ and $\tau(g)=0$ for $|g|\ge R$. Set
\begin{equation}\label{eq:tent}
u_R(y) \;:=\; \sup_{g\in G}\ \tau(g)\,
\Bigl(1-\frac{d\bigl(g^{-1}\cdot y,\;x_0\bigr)}{r}\Bigr)^{\!+},
\qquad y\in M,
\end{equation}
where $t^+:=\max\{t,0\}$. Since $\tau$ is supported on the finite ball
$\{|g|\le R-1\}$, the supremum in \eqref{eq:tent} is a maximum of
finitely many continuous functions of $y$ (together with the constant
$0$), so $u_R\in\CzM$ and $0\le u_R\le 1$.
\end{construction}

\begin{lemma}\label{lem:tent-properties}
The function $u_R$ of Construction~\ref{con:tent} satisfies:
\begin{enumerate}
\item[(i)] $u_R(x_0)=1$;
\item[(ii)] $u_R(z^*)=0$;
\item[(iii)] $\normz{\,u_R\circ\alpha(s)-u_R\,}\le \dfrac{1}{R}$ for
every $s\in S$; consequently $\norm{\delta u_R}\le 1/R$;
\item[(iv)] $\dist_{C^0}\bigl(u_R,\Inv(\alpha)\bigr)\ge \dfrac12$.
\end{enumerate}
\end{lemma}

\begin{proof}
\begin{enumerate}
\item[(i)] Take $g=e$ in \eqref{eq:tent}. 

\item[(ii)] Let $g\in G$ with $\tau(g)>0$, i.e.\ $|g|\le R-1$. Then
$g^{-1}\cdot z^*$ is an orbit point of the form $h\cdot x_0$ with
$|h|\le |g|+ (R+1)\le 2R$ (write $z^*=g^*\cdot x_0$ with
$|g^*|=R+1$, by \eqref{eq:orb-vs-word}, and $h=g^{-1}g^*$). Hence
$g^{-1}\cdot z^*\in E_R$. Moreover $g^{-1}\cdot z^*\ne x_0$, since otherwise $z^*=g\cdot x_0$ and $d_{\mathrm{orb}}(x_0,z^*)\le|g|\le R-1<R+1$, which is a contradiction. Therefore $d\bigl(g^{-1}\cdot z^*,x_0\bigr)\ge
\varepsilon_R> r$, so every term in \eqref{eq:tent} vanishes at
$y=z^*$.

\item[(iii)] Fix $s\in S$ and $y\in M$. Substituting $g=sh$,
\[
u_R(s\cdot y)
= \sup_{g\in G}\ \tau(g)\Bigl(1-\tfrac{d(g^{-1}s\cdot y,\,x_0)}{r}
\Bigr)^{+}
= \sup_{h\in G}\ \tau(sh)\Bigl(1-\tfrac{d(h^{-1}\cdot y,\,x_0)}{r}
\Bigr)^{+} .
\]
Since word length is $1$-Lipschitz under left multiplication by a
generator, $\bigl|\tau(sh)-\tau(h)\bigr|\le 1/R$ for all $h$, and
since $0\le\bigl(1-d(h^{-1}\cdot y,x_0)/r\bigr)^{+}\le 1$,
\[
\bigl| u_R(s\cdot y)-u_R(y) \bigr|
\;\le\; \sup_{h\in G}\ \bigl|\tau(sh)-\tau(h)\bigr|
\;\le\; \frac1R .
\]

\item[(iv)] Every $v\in\Inv(\alpha)$ is constant on orbits, so
$v(x_0)=v(z^*)=:a$. Now by (i) and (ii),
$\normz{u_R-v}\ge\max\{|1-a|,\,|a|\}\ge\tfrac12$.
\end{enumerate}
\end{proof}

Now we are in a position to prove Theorem A.

\begin{proof}[Proof of Theorem A]
If $\alpha(G)$ is finite, apply Lemma~\ref{lem:finite-stable}.
Conversely, suppose $\alpha(G)$ is infinite and $\cob$ is closed. Let
$C>0$ be a constant as in Lemma~\ref{lem:open-mapping}. For each $R\ge 2$,
Construction~\ref{con:tent} and Lemma~\ref{lem:tent-properties} give
$u_R\in\CzM$ with
\[
\tfrac12 \;\le\; \dist_{C^0}\bigl(u_R,\Inv(\alpha)\bigr)
\;\le\; C\,\norm{\delta u_R} \;\le\; \frac{C}{R}
\;\xrightarrow[R\to\infty]{}\;0,
\]
a contradiction. Hence $\cob$ is not closed.
\end{proof}

\begin{remark}\label{rem:no-disjointness}
It is worth comparing our setting with that one in \cite{Koc13}. There, for $G=\Z$, the almost
invariant function is a sum of bumps carried by the disjoint
iterates $f^j(B_n)$, $|j|<2^n$, of a small ball, with the crucial
transported profile $d(f^{-j}(x),x_n)$. Disjointness of the iterates
is what requires choosing a point of period strictly greater then $2^n$. For general group $G$ the translates $g\cdot B(x_0,r)$, $|g|<R$, may overlap heavily (large
point stabilizers, non-free actions), and no choice of $r$ can
separate them. Formula \eqref{eq:tent} sidesteps this entirely: the
supremum of exactly transported profiles is automatically
$\alpha$-equivariantly compatible, no disjointness is used anywhere,
and the only geometric input is the separation of the finitely many
distinct orbit points in $E_R$. In particular, specializing to
$G=\Z$, \eqref{eq:tent} also gives a slightly shorter proof of
Kocsard's original theorem.
\end{remark}

\begin{remark}[Quantitative instability]\label{rem:quantitative}
Actually, our proof shows more: if $\alpha(G)$ is infinite, the inverse $\bar\delta^{\,-1}$ fails to be bounded at a rate controlled by the
diameters of orbit Schreier graphs. Since
$\diam\ge\log_{|S|+1}\#(\text{orbit})$, even actions with the fastest
possible orbit growth (e.g.\ boundary actions of free groups, or free
subgroups of $SO(3)$ acting on the sphere $\mathbb S^2$) admit functions with
$\norm{\delta u}\le 1/R$ at uniform distance greater or equal to $1/2$ from $\Inv(\alpha)$, merely with the base scale $r=r(R)$ deteriorating.
See Section \ref{sec:gap}.
\end{remark}

\section{The closure of the coboundary space and the role of
amenability}\label{sec:amenable}

Theorem A shows that amenability of $G$ (or of the action) is
irrelevant for the closedness of $\cob$. Where amenability does
enter the picture is in the description of the
closure $\overline{\cob}$, i.e. of the complete system of
cohomological obstructions in the sense of \cite[\S(a)]{Koc13}.

For $G=\Z$, Moulin~Ollagnier and Pinchon \cite{MOP77} proved
\[
\overline{B(f,\CzM)}^{\,C^0}
= \Bigl\{ \phi\in\CzM \;:\; \int_M\phi\,d\mu=0,\ \
\forall\,\mu\in\mathcal M(f) \Bigr\},
\]
where $\mathcal M(f)\neq\emptyset$ (by Krylov-Bogolyubov theorem) is the set of
$f$-invariant Borel probability measures. The general statement has the
following shape.

\begin{proposition}\label{prop:annihilator}
For any finitely generated $G$ and any action $\alpha$,
\[
\overline{\cob}
= \Bigl\{ c\in\coc \;:\; \textstyle\sum_{s\in S}\int_M c(s,\cdot)\,
d\lambda_s = 0 \ \text{ for every } (\lambda_s)_{s\in S}\in
\mathcal{A}(\alpha) \Bigr\},
\]
where $\mathcal A(\alpha)$ denotes the set of tuples of finite signed
Borel measures on $M$ satisfying
$\sum_{s\in S}\bigl(\alpha(s)^{-1}_{*}\lambda_s-\lambda_s\bigr)=0$.
\end{proposition}

\begin{proof}
By the Riesz and Hahn-Banach theorems, the closure of the subspace
$\cob\subset\coc\subset\CzM^S$ is the joint kernel of the bounded
functionals annihilating it, and a tuple $(\lambda_s)$ annihilates all
$\delta u$ if and only if $\sum_s\int (u\circ\alpha(s)-u)\,d\lambda_s=0$ for all $u\in\CzM$,
which is precisely the stated condition.
\end{proof}

The content of \cite{MOP77} is that for $G=\Z$ the implicit set
$\mathcal A(\alpha)$ reduces to invariant measures (via the Jordan
decomposition, which is preserved by $f_*$). For general $G$, each
$\alpha$-invariant probability $\mu\in\mathcal M(\alpha)$ and each
tuple of weights produces elements of $\mathcal A(\alpha)$, giving the
familiar obstructions: for $\mu\in\mathcal M(\alpha)$ the map
\[
\rho_\mu(c)\colon\; g\;\longmapsto\;\int_M c(g,\cdot)\,d\mu
\]
is a homomorphism $G\to\R$ (this is immediate from \eqref{eq:cocycle} and
invariance), and $\rho_\mu(c)=0$ for every coboundary. When the action is amenable in the sense that $\mathcal M(\alpha)\ne\emptyset$
- automatic when the group $G$ is amenable - these obstructions are
non-trivial and one expects, in analogy with \cite{MOP77}, positive answer to the following question.

\begin{question}\label{q:amenable-closure}
For $G$ finitely generated and amenable, does one have
\[
\overline{\cob}
\;=\; \bigl\{\, c\in\coc \;:\; \rho_\mu(c)=0,\ \
\forall\,\mu\in\mathcal M(\alpha) \,\bigr\}\ ?
\]

\end{question}

For non-amenable actions, by contrast, $\mathcal M(\alpha)$ may be
empty (e.g.\ boundary actions of non-elementary hyperbolic groups, or
$\mathrm{PSL}_2(\R)$-type actions on the circle), all homomorphism
obstructions then vanish, and the following question is expected to have affirmative answer.

\begin{question}\label{q:dense}
Let $\alpha$ be an action with $\mathcal M(\alpha)=\emptyset$ (say, the
boundary action of a free group on its Cantor boundary). Is
$\overline{\cob}=\coc$, i.e.\ is every continuous cocycle a uniform
limit of coboundaries? Equivalently, is $\mathcal A(\alpha)$ reduced to
tuples annihilating all cocycles?
\end{question}

Thus the genuine amenable/non-amenable dichotomy in this story concerns
which obstructions cut out $\overline{\cob}$ and not whether
$\cob$ itself is closed, which by Theorem A is decided by finiteness
alone.

\section{Why spectral gaps do not restore $C^0$-stability}
\label{sec:gap}

Before the proof in Section \ref{sec:unstable} one might have conjectured
that suitable non-amenable actions are cohomologically $C^0$-stable,
i.e.\ that they are counterexamples to the ``only periodic'' half of
Theorem A. We explain why the natural candidates fail, since the
mechanism is instructive.

\subsection{Rotation actions with spectral gap}
Let $F\le SO(3)$ be a free group of rank $2$ whose generators act on
$\mathbb S^2$ with a spectral gap, i.e.\ the averaging (Markov) operator
$P=\frac{1}{2|S|}\sum_{s\in S}\,u\circ s$ satisfies
$\|P u\|_{L^2}\le\theta\|u\|_{L^2}$ on
$L^2_0(\mathbb S^2,\mathrm{Leb})$ for some $\theta<1$. The existence of such
$F$ goes back to Drinfeld \cite{Dri84}, and holds for all generic and
all algebraic generators by Bourgain-Gamburd \cite{BG08}. The gap
implies the coercive inequality
\[
\Bigl\| u - \textstyle\int u \Bigr\|_{L^2}
\;\le\; C \max_{s\in S}\ \norm{u\circ s-u}_{L^2},
\]
which is exactly $L^2$-version of \eqref{eq:gap-ineq} (with
$\Inv=\{\text{constant functions}\}$, by ergodicity). If the same inequality held in
the uniform norm, Lemma~\ref{lem:open-mapping} would be contradicted
nowhere and the action would be a stable infinite action. Theorem A
says it does not hold, and Construction~\ref{con:tent} exhibits the
counterexample: $u_R$ is a tent, in the word metric, over the
ball $\{g\cdot x_0: |g|<R\}$ of a single orbit, thickened at a tiny
scale $r$. Its uniform ``gradient'' is $\le 1/R$, while its
oscillation along the orbit is $1$. There is no conflict with the
spectral gap: $u_R$ is supported on a set of Lebesgue measure
$O\bigl((|S|+1)^{R} r^{2}\bigr)$, which is made arbitrarily small by
the choice of $r$, so $\|u_R\|_{L^2}$ is negligible and the $L^2$
inequality is comfortably satisfied. In other words expansion and spectral gaps are $\ell^2$-phenomena. The uniform norm does not see measure, and every graph of large
diameter carries $\ell^\infty$-slowly-varying non-constant (tent) functions. \\

This is the same reason why expander graphs, despite their gap, admit
$1/\!\diam$-Lipschitz functions of oscillation $1$ (distance
functions), the diameter being $\asymp\log$ of the size.

\subsection{Hyperbolic-type actions}
One might instead hope that hyperbolicity yields $C^0$-stability, by
analogy with Liv\v sic theory. But this fails already for $G=\Z$.
Anosov diffeomorphisms are cohomologically H\"older-stable
\cite{Liv72} and even $C^r$-stable for $r\in[2,\infty]$
\cite{dlLMM86}, yet they are aperiodic, hence cohomologically
$C^0$-unstable by Kocsard's theorem. The mechanism is the same,
Liv\v sic's proof produces a transfer function with a modulus of
regularity, and the closed-range inequality holds between H\"older
norms; it degenerates as the H\"older exponent tends to $0$. Boundary
actions of hyperbolic groups, which enjoy strong topological stability
properties in the sense of structural stability 
are therefore not expected to be (and by Theorem A are not) cohomologically $C^0$-stable either.

\subsection{Where the gap does give closed range}
Closedness of the coboundary space is a
property of the dynamics paired coupled with function space:
\begin{itemize}
\item in $C^0$: closed if and only if the acting group is finite (Theorem A) -
finiteness is the only mechanism;
\item in H\"older/$C^r$ spaces: hyperbolicity gives closedness
\cite{Liv72,dlLMM86};
\item in $L^2(\mu)$ of an invariant measure: closedness of
$\delta(L^2)$ is equivalent to the absence of almost invariant
vectors in $L^2_0$, hence holds for actions with spectral gap and, for
groups with Kazhdan's property (T), for all their
measure-preserving actions;
\item in $C^\infty$: for circle diffeomorphisms, arithmetic
(Diophantine) conditions decide \cite{AK11}. See the next Section \ref{sec:smooth}.
\end{itemize}

\section{$C^0$ versus $C^\infty$: commuting circle diffeomorphisms
and the simultaneous Diophantine condition}\label{sec:smooth}

Theorem A shows that the $C^0$-category does not see the arithmetic. Any $\Z^k$-action on $S^1$ containing a diffeomorphism of irrational
rotation number has infinite image and is therefore $C^0$-unstable,
regardless of any Diophantine property of the rotation numbers. The
smooth category behaves in the opposite way. For a single map, the
definitive statement is the following.

\begin{theorem}[Avila-Kocsard \cite{AK11}]\label{thm:AK}
Let $f\in\Diff^\infty_+(\mathbb S^1)$ have no periodic points, with rotation
number $\rho=\rho(f)$. $f$ is cohomologically $C^\infty$-stable, i.e.\ $B(f,C^\infty(S^1))$ is closed in $C^\infty(\mathbb S^1)$ if and
only if $\rho$ is Diophantine. Moreover, the space of $f$-invariant
distributions is spanned by the unique invariant measure $\mu_f$, so in the stable case $H^1$ is one-dimensional.
\end{theorem}

For $\Z^k$-actions the relevant arithmetic condition is the one
discovered by Moser \cite{Mos90} in his study of the local version of
the problem of smooth linearization of commuting circle maps.

\begin{definition}[SDC]\label{def:SDC}
A vector $\rho=(\rho_1,\dots,\rho_k)\in\R^k$ satisfies the
simultaneous Diophantine condition $\mathrm{SDC}(\gamma,\tau)$,
$\gamma>0$, $\tau>0$, if
\[
\max_{1\le i\le k}\ \| q\,\rho_i \|_{\Z}
\;\ge\; \frac{\gamma}{|q|^{\tau}},
\qquad \forall\, q\in\Z\smallsetminus\{0\},
\]
where $\|\cdot\|_\Z$ is the distance to the nearest integer. We say
$\rho$ is SDC if it satisfies $\mathrm{SDC}(\gamma,\tau)$ for some
$\gamma,\tau$.
\end{definition}

If some $\rho_i$ is Diophantine then $\rho$ is trivially SDC. The
interesting SDC vectors are tuples of Liouville numbers whose
denominators of good rational approximation never synchronize.
Conversely, $\rho$ fails SDC exactly when all coordinates admit
anomalously good rational approximations along one common sequence of
denominators. \\

Two known results frame the conjecture below. First, Fayad and Khanin
proved the global form of Moser's conjecture from \cite{Mos90}.

\begin{theorem}[Fayad-Khanin \cite{FK09}]\label{thm:FK}
If $f_1,\dots,f_k\in\Diff^\infty_+(\mathbb S^1)$ commute and their rotation
vector $\rho=(\rho(f_1),\dots,\rho(f_k))$ is SDC, then the action is
simultaneously $C^\infty$-conjugate to the action by the rotations
$R_{\rho_1},\dots,R_{\rho_k}$.
\end{theorem}

Second, the linear model is completely computable.

\begin{proposition}[Linear model]\label{prop:linear}
Let $\rho\in\R^k$ have at least one irrational coordinate and let
$\alpha_\rho\colon\Z^k\to\Diff^\infty_+(\mathbb S^1)$ be the action by the
rotations $R_{\rho_i}$. Then $\alpha_\rho$ is cohomologically
$C^\infty$-stable (the space of smooth coboundaries is closed in the space of smooth
cocycles with its Fr\'echet topology) if and only if $\rho$ is SDC. Moreover, in
the stable case
\[
B\bigl(\alpha_\rho, C^\infty\bigr)
= \Bigl\{ (\phi_1,\dots,\phi_k) \ \text{smooth cocycle} :
\int_{S^1}\phi_i\,dx = 0,\ i=1,\dots,k \Bigr\}.
\]
\end{proposition}

\begin{proof}[Sketch of proof]
Passing to Fourier coefficients, the compatibility relations of
Remark~\ref{rem:closed-in-product} read
$\bigl(e^{2\pi i q\rho_i}-1\bigr)\widehat{\phi_j}(q)
=\bigl(e^{2\pi i q\rho_j}-1\bigr)\widehat{\phi_i}(q)$ for all $q$.
Assuming the means vanish, define for $q\ne0$
\[
\widehat{u}(q) :=
\frac{\widehat{\phi_{i(q)}}(q)}{e^{2\pi i q\rho_{i(q)}}-1},
\qquad
i(q):=\operatorname*{arg\,max}_{1\le i\le k}\ \|q\rho_i\|_\Z ,
\]
which is well defined and independent of admissible choices by
compatibility. Under SDC, $|e^{2\pi i q\rho_{i(q)}}-1|\ge
4\gamma|q|^{-\tau}$, so the division loses only finitely many
derivatives and $u\in C^\infty$, with tame estimates. Hence the
coboundary space equals the closed space of zero-mean cocycles. If SDC
fails, pick $q_n\to\infty$ with $\max_i\|q_n\rho_i\|_\Z\le
|q_n|^{-n}$, and build a zero-mean smooth cocycle supported on the
frequencies $\pm q_n$ whose (unique, since some $\rho_i$ is
irrational) formal transfer function has Fourier coefficients growing
faster than any polynomial. Its truncations are coboundaries
converging to it in $C^\infty$, so the coboundary space is not closed.
Compare \cite[ 2]{Kat01} for the classical case of tori.
\end{proof}

Since cohomological $C^\infty$-stability is invariant under smooth
conjugacy, Theorem~\ref{thm:FK} and Proposition~\ref{prop:linear}
immediately give: an SDC action of $\Z^k$ by smooth circle
diffeomorphisms is cohomologically $C^\infty$-stable. We conjecture
that SDC is also necessary.

\begin{conjecture}\label{conj:SDC}
Let $\alpha\colon\Z^k\to\Diff^\infty_+(S^1)$ be generated by commuting
diffeomorphisms $f_1,\dots,f_k$ such that at least one $\rho(f_i)$ is
irrational, and let $\rho=(\rho(f_1),\dots,\rho(f_k))$. Then $\alpha$
is cohomologically $C^\infty$-stable if and only if $\rho$ is SDC.
\end{conjecture}

Some remarks on the status of Conjecture~\ref{conj:SDC}:

\begin{enumerate}
\item ($\Leftarrow$) holds, as explained above, by
\cite{FK09} + Proposition~\ref{prop:linear}.
\item ($\Rightarrow$) holds for linearizable actions: if
$\alpha$ is $C^\infty$-conjugate to $\alpha_\rho$ and $\rho$ is not
SDC, Proposition~\ref{prop:linear} gives instability. For $k=1$ this,
combined with Herman-Yoccoz theory and the invariant-distribution
analysis, is exactly the content of Theorem~\ref{thm:AK}
(\cite{AK11}), whose proof however does not pass through
linearization in the Liouville case.
\item The genuinely open case is therefore: non-SDC rotation vector
and non-linearizable (in particular, singular) conjugacy. Note
that failure of SDC forces every $\rho_i$ to be Liouville, and
Anosov-Katok-type constructions produce $\Z^k$-actions with
prescribed non-SDC Liouville rotation vectors whose conjugacy to the
linear model is purely singular. See \cite{FK09} and \cite{P26} for such examples with purely singular simultaneous linearization. For these actions the linear computation is unavailable, and
instability should instead be detected by exhibiting invariant
distributions of higher order, or almost-coboundaries with divergent
transfer functions, in the spirit of \cite{AK11}. Such examples are
the natural test cases for the conjecture.
\item Higher-dimensional analogue. For $\Z^m$-actions on $\T^d$
generated by commuting diffeomorphisms close to a simultaneously
Diophantine family of translations a
perturbative rigidity/classification theorem is proved in \cite{P21}, the KAM scheme being governed precisely by an SDC condition on the rotation set. The
corresponding cohomological-stability conjecture on $\T^d$ reads: a
$\Z^m$-action by diffeomorphisms of $\T^d$, $C^\infty$-conjugate (or
a priori just semi-conjugate) to an ergodic translation action with
rotation set $(\rho^{(1)},\dots,\rho^{(m)})\in(\R^d)^m$, is
cohomologically $C^\infty$-stable if and only if
\[
\exists\,\gamma,\tau>0:\quad
\max_{1\le i\le m}\ \| \langle q,\rho^{(i)}\rangle \|_{\Z}
\;\ge\; \frac{\gamma}{|q|^{\tau}},
\qquad \forall\, q\in\Z^{d}\smallsetminus\{0\},
\]
the linear case again being a Fourier computation as in
Proposition~\ref{prop:linear} (cf.\ \cite[ 2]{Kat01} for $m=1$).
\item Finally, Theorem A supplies the $C^0$-counterpart for every one
of these actions: infinite image, hence $C^0$-unstable. The pair
(Theorem A, Conjecture~\ref{conj:SDC}) thus exhibits, on the same
class of systems, the transition from a purely algebraic dichotomy in
$C^0$ to a purely arithmetic one in $C^\infty$. The finite
differentiability classes $C^r$, $1\le r<\infty$, interpolating
between the two, are essentially unexplored even for $k=1$ (cf.\ the
invariant distributions of low regularity of \cite{NT13}).
\end{enumerate}

We conclude by stating few questions that would be worth answering.

\begin{enumerate}
\item (\emph{Non-finitely generated groups.}) Theorem A uses finite
generation through the word metric and through the finiteness of the
set of subgroups of bounded index. For a countable, non-finitely
generated $G$ (e.g.\ $G=\Q$), $\coc$ carries the topology of uniform
convergence on each $g$. Does the stability still holds if and only if the acting group has finite image?
\item (\emph{Flows and topological groups.}) For $\R$-actions
(continuous flows) with cocycles $c\colon\R\times M\to\R$, the tent
construction has an obvious continuous-time analogue. Does the
stability hold if and only if the flow factors through a compact group
(e.g.\ periodic flows). 
\item (\emph{Closure of the coboundaries.}) Questions
\ref{q:amenable-closure} and \ref{q:dense}: identify
$\overline{\cob}$ for amenable actions via invariant measures, and
see whether coboundaries are uniformly dense in $\coc$ for actions
without invariant measures.
\item (\emph{Quantitative version.}) Estimate the distortion
$\norm{\bar\delta^{-1}\!\restriction_{\text{finite-dim.\ subspaces}}}$
in terms of orbit growth. Remark~\ref{rem:quantitative} suggests a
$\log$-law for actions of exponential orbit growth versus a power law
in the abelian case.
\item (\emph{The smooth conjecture.}) Conjecture~\ref{conj:SDC},
starting with the purely singular non-SDC examples of \cite{P26}.
\end{enumerate}


\begin{thebibliography}{MOP77}

\bibitem[AK11]{AK11}
A.~Avila and A.~Kocsard,
\emph{Cohomological equations and invariant distributions for minimal
circle diffeomorphisms},
Duke Math.\ J.\ \textbf{158} (2011), no.~3, 501-536.

\bibitem[BG08]{BG08}
J.~Bourgain and A.~Gamburd,
\emph{On the spectral gap for finitely-generated subgroups of
$SU(2)$},
Invent.\ Math.\ \textbf{171} (2008), no.~1, 83-121.

\bibitem[dlLMM86]{dlLMM86}
R.~de~la~Llave, J.~M.~Marco and R.~Moriy\'on,
\emph{Canonical perturbation theory of Anosov systems and regularity
results for the Liv\v sic cohomology equation},
Ann.\ of Math.\ (2) \textbf{123} (1986), no.~3, 537-611.

\bibitem[Dri84]{Dri84}
V.~G.~Drinfel'd,
\emph{Finitely-additive measures on $S^2$ and $S^3$, invariant with
respect to rotations},
Funct.\ Anal.\ Appl.\ \textbf{18} (1984), no.~3, 245-246.

\bibitem[FK09]{FK09}
B.~Fayad and K.~Khanin,
\emph{Smooth linearization of commuting circle diffeomorphisms},
Ann.\ of Math.\ (2) \textbf{170} (2009), no.~2, 961-980.

\bibitem[Kat01]{Kat01}
A.~Katok,
\emph{Cocycles, cohomology and combinatorial constructions in ergodic
theory}, in collaboration with E.~A.~Robinson, Jr.,
Proc.\ Sympos.\ Pure Math., vol.~69, Amer.\ Math.\ Soc., Providence,
RI, 2001, pp.~107-173.

\bibitem[Koc13]{Koc13}
A.~Kocsard,
\emph{On cohomological $C^0$-(in)stability},
Bull.\ Braz.\ Math.\ Soc.\ (N.S.) \textbf{44} (2013), no.~3.
DOI:10.1007/s00574-013-0023-9.

\bibitem[Liv72]{Liv72}
A.~N.~Liv\v sic,
\emph{Cohomology of dynamical systems},
Izv.\ Akad.\ Nauk SSSR Ser.\ Mat.\ \textbf{36} (1972), 1296-1320.

\bibitem[Mos90]{Mos90}
J.~Moser,
\emph{On commuting circle mappings and simultaneous Diophantine
approximations},
Math.\ Z.\ \textbf{205} (1990), no.~1, 105-121.

\bibitem[MOP77]{MOP77}
J.~Moulin~Ollagnier and D.~Pinchon,
\emph{Syst\`emes dynamiques topologiques. I. \'Etude des limites de
cobords},
Bull.\ Soc.\ Math.\ France \textbf{105} (1977), no.~4, 405-414.

\bibitem[NT13]{NT13}
A.~Navas and M.~Triestino,
\emph{On the invariant distributions of $C^2$ circle diffeomorphisms
of irrational rotation number},
Math.\ Z.\ \textbf{274} (2013), 315-321.

\bibitem[P21]{P21}
B.~Petkovi\'c,
\emph{Classification of perturbations of Diophantine $\Z^m$ actions
on tori of arbitrary dimension},
Regul.\ Chaotic Dyn.\ \textbf{26} (2021), no.~6, 700-716.

\bibitem[P26]{P26}
B.~Petkovi\'c,
\emph{Commuting circle diffeomorphisms with non-simultaneously
Diophantine rotation numbers and purely singular linearization},
preprint (2026).

\end{thebibliography}
\end{document}